\documentclass[11pt]{article}

\usepackage[margin=1in]{geometry}

\usepackage{amsthm}
\usepackage{amsmath}
\usepackage{amssymb}
\usepackage{enumerate}
 
\newtheorem{theorem}{Theorem}
\newtheorem{lemma}{Lemma}
\newtheorem{remark}{Remark}
\newtheorem{definition}{Definition}
\newtheorem{corollary}{Corollary}
\newtheorem{conjecture}{Conjecture}

\newtheorem{proposition}{Proposition}

\newtheorem{question}{Question}

\newcommand{\Tr}{\operatorname{Tr}}

\newcommand{\R}{\mathbb{R}}

\begin{document}

\title{Stability of the Poincar\'e--Korn inequality
}

\author{%
  Thomas~A.~Courtade\footnote{University of California, Berkeley, United States, Department of Electrical Engineering and Computer Sciences.\\
    courtade@berkeley.edu}
  \and 
  Max~Fathi\footnote{Université Paris Cité and Sorbonne Université, CNRS, Laboratoire Jacques-Louis Lions and Laboratoire de Probabilit\'es, Statistique et Mod\'elisation, F-75013 Paris, France;  and
  DMA, École normale supérieure, Université PSL, CNRS, 75005 Paris, France; and
 Institut Universitaire de France.\\ mfathi@lpsm.paris }}

\date{May 2nd 2024}
\maketitle
 
\abstract{We resolve a question of Carrapatoso et al.~\cite{CDHMM} on Gaussian optimality for the sharp constant in Poincar\'e-Korn inequalities, under a moment constraint. We also prove stability, showing that measures with near-optimal constant are quantitatively close  to   standard Gaussian.}

\section{Introduction and Main Result}

Let $\mu$ be a centered Borel probability measure on $\mathbb{R}^n$,  $n\geq 2$.  Let $\mathcal{A}$ denote the set of  antisymmetric linear maps from $\mathbb{R}^n$ to itself.  That is, 
$$
\mathcal{A} := \{ x \mapsto A x~; A \in   M_{n\times n}(\mathbb{R})  , A = -A^T\}. 
$$
Further define the linear space of vector-valued functions 
$$
\mathcal{C}:= \{ u : \mathbb{R}^n \to  \mathbb{R}^n~;~\mbox{$u$ differentiable, $\int u d\mu = 0$, and $\|\nabla_{s}u\|_{L^2(\mu)}^2 < \infty$}\}, 
$$
where $\nabla_{s} u := \frac{1}{2}((\nabla u) - (\nabla u)^T)$ is the  symmetrized gradient of the vector-valued function $u : \mathbb{R}^n \to \mathbb{R}^n$.   Since $\mu$ is centered\footnote{The persistent centering assumption   comes without any   loss of generality, and is only made for convenience.},  $\mathcal{A}$ is a closed linear subspace of $\mathcal{C}$.  

\begin{definition}
A centered Borel probability measure $\mu$ on $\mathbb{R}^n$ satisfies a {\bf Poincar\'{e}--Korn inequality} with constant $C$ if 
\begin{align}
\inf_{a\in \mathcal{A}}\|u - a \|^2_{L^2(\mu)} \leq 2 C \|\nabla_{s}u\|_{L^2(\mu)}^2, ~~ \mbox{for all   $u \in \mathcal{C}$.} \label{ineq:PK}
\end{align}
The {\bf Poincar\'{e}--Korn constant} associated to $\mu$, denoted $C_{PK}(\mu)$, is the smallest constant $C$ such that \eqref{ineq:PK} holds.
\end{definition}

This type of inequality was introduced in \cite{CDHMM}. It is inspired by Poincar\'e inequalities, which control variances of scalar valued functions by the $L^2$ norm of their gradient, and the Korn inequality from continuous mechanics, which controls the $L^2$ norm of the gradient of a vector field satisfying some boundary condition by its symmetric part. Both inequalities have found many applications in analysis. They were both originally introduced for uniform measures on domains, but can be extended to general probability densities. In this form, Poincar\'e inequalities are
\begin{equation}
\int{f^2d\mu} - \left(\int{fd\mu}\right)^2 \leq C_P(\mu)\int{|\nabla f|^2d\mu}, \hspace{3mm} \forall f : \R^d \longrightarrow \R^d, 
\end{equation}
where $C_P(\mu)$ is the Poincar\'e constant of $\mu$, and the right-hand side is to be understood as $+\infty$ if it is not well-defined for the function $f$.   Korn inequalities are of the form
\begin{equation}
\inf_{a \in \mathcal{A}} \|\nabla(u - a)\| ^2_{L^2(\mu)} \leq C_K  \|\nabla_{s}u\|_{L^2(\mu)}^2, ~~ \mbox{for all   $u \in \mathcal{C}$.}
\end{equation}

Poincar\'e inequalities are now a very classical tool in probability and functional analysis, and applications include concentration of measure inequalities and rates of convergence to equilibrium for Markov processes. We refer to the monograph \cite{BGL14} for background and many developments. On the other hand, classical Korn inequalities are a tool in kinetic theory and fluid mechanics, going back to \cite{Kor06}. We shall make no attempt to survey the vast literature, and refer to \cite{Hor95} for some background. Best constants were investigated for example in \cite{LM16}. Weighted Korn inequalities were recently introduced in \cite{CDHMM}, motivated by hypocoercivity problems in kinetic theory. 

Our definition of the Poincar\'e--Korn constant differs from that in \cite{CDHMM} by a factor of 2.  This is done to give unit normalization with respect to the standard Gaussian measure $\gamma$, defined by
$$
d\gamma(x) := \frac{1}{(2\pi)^{n/2}}e^{-|x|^2/2} dx, ~~~~x\in \R^n.
$$
Probability measures with sufficiently regular potentials admit a finite Poincar\'e--Korn constant if they satisfy a  Poincar\'e inequality with finite constant.
\begin{proposition}\cite[Theorem 1]{CDHMM} 
We have  $C_{PK}(\gamma)=1$ in any dimension $n \geq 2$.    Moreover, if   a centered probability measure  with density $d\mu = e^{-\phi}dx$ of class $C^2$   satisfies
$$\forall \epsilon > 0 , \exists C_{\epsilon}>0 : ~~\|\nabla^2 \phi(x)\|^2 \leq \epsilon |\nabla \phi(x)|^2 + C_{\epsilon},\hspace{3mm} \forall x\in \R^n,$$
then $C_{P}(\mu) < \infty \Rightarrow C_{PK}(\mu)<\infty$. 
\end{proposition}


In \cite{CDHMM}, the following conjecture is proposed regarding the rigidity of the Poincar\'e--Korn inequality:
\begin{conjecture}\label{conj:CDHMM}
If $d\mu = e^{-\phi}dx$ is centered and isotropic (i.e., $\int x d\mu = 0$ and $\int x x^T d\mu = \operatorname{Id}$), and satisfies $\nabla^2 \phi \geq \operatorname{Id}$, then $C_{PK}(\mu)\geq C_{PK}(\gamma)$, with equality   only if $\mu = \gamma$.  
\end{conjecture}

It turns out that the conjectured statement  is indeed true, but the hypotheses are too strong to capture   salient rigidity properties of the Poincar\'e--Korn constant.  Namely, the following can be derived as a consequence of known stability results for the Bakry--\'Emery theorem on $\mathbb{R}^n$ \cite{CZ17}, or using Caffarelli's contraction theorem (see Appendix A). We make no claim of originality for this statement, which was  known in some communities. 
\begin{proposition}\label{prop_conj_triviale}
If $d\mu = e^{-\phi}dx$ is centered, isotropic,  and satisfies $\nabla^2 \phi \geq \operatorname{Id}$, then $\mu = \gamma$.  
\end{proposition}
Therefore, to study rigidity of the Poincar\'e--Korn inequality, the assumption of a uniformly convex potential in Conjecture \ref{conj:CDHMM} should be replaced by something else.  A natural choice is a  moment assumption, which we now define.
\begin{definition}[Moment Assumption]
We say that  $\mu$ satisfies the {\bf moment assumption} if, for all $1\leq i ,j,k\leq n$,
\begin{align*}
\int x_i d\mu &= \int x_i d\gamma = 0;\\
\int x_i x_j d\mu &= \int x_i x_j d\gamma = \delta_{ij};\\
\int x_i x_j x_k d\mu &= \int x_i x_j x_k d\gamma  = 0;
\end{align*}
and, when $i\neq j$, 
\begin{align*}
\int (x_i^2+ x_j^2) x_j^2 d\mu &= \int (x_i^2+ x_j^2) x_j^2  d\gamma = 4.
\end{align*}
\end{definition}
\begin{remark}
The first two lines of the moment assumption correspond to $\mu$ being centered and isotropic.  The third and fourth lines in the moment assumption ensure that $\mu$ and $\gamma$ share  mixed third moments and select mixed fourth moments, respectively.    
\end{remark}

There are many interesting probability measures that satisfy the moment assumption.  For example, any product measure whose individual factors share moments up to order 4 with the standard normal will satisfy the moment assumption (and so will mixtures of these measures, and so forth...).  Thus, a nontrivial reformulation of Conjecture \ref{conj:CDHMM} is as follows:
\begin{question}\label{qu:main}
If $\mu$ satisfies the moment assumption, is the lower bound $C_{PK}(\mu)\geq C_{PK}(\gamma)$ true, with equality   only if $\mu = \gamma$?  
\end{question}

The   moment assumption is motivated by the  form of extremal functions in the Gaussian Poincar\'e-Korn inequality. In particular, by considering these as test functions in the Poincar\'e--Korn inequality for $\mu$,  the inequality $C_{PK}(\mu)\geq C_{PK}(\gamma)$   is a consequence of  the moment assumption  (see Proposition \ref{prop:MomentImplications} in the sequel).  So, it is the rigidity phenomenon that is interesting. We remark that the work of Serres \cite{Ser23} already highlights that given a reference measure satisfying a Poincaré inequality with known sharp constant and extremal function, it is possible to study stability of functional inequalities within classes of measures for which the moments of the extremal function match with those under the reference measure. 

In this note, we resolve  Question \ref{qu:main} in the affirmative, and further establish quantitative stability of the Poincar\'e--Korn constant.  Such stability results on functional inequalities have been the subject of  some recent attention in analysis. For example, there have been many results on stability for sharp functions in classical functional inequalities, including Sobolev inequalities and isoperimetric inequalities, see \cite{Fig13, Fig23} for surveys. Stability results under moment constraints have been studied for Poincar\'e inequalities \cite{Ut89, CPU94, CFP18, Ser23}, eigenvalues of diffusion operators \cite{Ser23b}, as well as fractional \cite{AH} and free \cite{CFM} Poincar\'e inequalities. These have been obtained using the combination of Stein's method and variational arguments that we shall use here. In another direction, there have been stability results for sharp constants under convexity or curvature assumptions \cite{CMS, CF20, FGS, MO}. 

Our stability result is with respect to   the Zolotarev distance of order 2, which controls the same topology as the more familiar $W_2$ Kantorovich--Wasserstein distance \cite{BH00}.  
\begin{definition}
For two Borel probability measures $\mu,\nu$ on $\mathbb{R}^n$, the {\bf Zolotarev distance of order 2} is defined as
$$
d_{Zol,2}(\mu,\nu) := \sup_{\sup_{x} \|\nabla^2 f(x)\|_{2}\leq 1}  \int f d\mu - \int f d\nu,
$$
where the supremum is over all $f: \mathbb{R}^n \to \mathbb{R}$ in $C^2$ with $\sup_{x\in \mathbb{R}^n } \|\nabla^2 f(x)\|_{2}\leq 1$. 
\end{definition}

Our main result is as follows.

\begin{theorem}[Stability]\label{thm:MainResult}
If $\mu$ satisfies the moment assumption, then $C_{PK}(\mu)\geq C_{PK}(\gamma)=1$, and 
$$
d_{Zol,2}(\mu,\gamma) \leq  c\,  n^2 \sqrt{C_{PK}(\mu) (C_{PK}(\mu)- C_{PK}(\gamma))},
$$
where $c$ is a universal constant.
\end{theorem}

The following    is now immediate, and affirmatively answers   Question \ref{qu:main}.  
\begin{corollary}[Rigidity]
If $\mu$ satisfies the moment assumption, then   $C_{PK}(\mu) \geq C_{PK}(\gamma)$, with equality only if $\mu = \gamma$. 
\end{corollary}

\textbf{Acknowledgments.} We thank Jean Dolbeault for telling us about this problem. TC acknowledges NSF-CCF~1750430, the hospitality of the Laboratoire de Probabilités, Statistique et Modélisation (LPSM) at the Université Paris Cité, and the  Invited Professor program of the Fondation Sciences Mathématiques de Paris (FSMP).  MF was supported by the Agence Nationale de la Recherche (ANR) Grant ANR-23-CE40-0003 (Project CONVIVIALITY). This work has also received support under the program ``Investissement d'Avenir" launched by the French Government and implemented by ANR, with the reference  ANR‐18‐IdEx‐0001 as part of its program Emergence.

 \section{Proof of Main Result}

\subsection{Notation}
For a vector-valued function $u:\mathbb{R}^n\to \mathbb{R}^n$, we define 
$$
\| u\|^2_{L^2(\mu)} := \int |u|^2 d\mu,
$$ 
where $|x |$ denotes the Euclidean length of $x\in \mathbb{R}^n$.  Likewise, for a   matrix-valued function $U : \mathbb{R}^n \to M_{n\times n}(\mathbb{R})$, we define 
$$
\| U \|^2_{L^2(\mu)} := \int U \cdot U d\mu,
$$ 
where `$\cdot$' will denote the scalar product between matrices  (i.e., $A\cdot B := \sum_{ij}[A]_{ij}[B]_{ij}$, for identically dimensioned matrices $A,B\in M_{m\times n}(\mathbb{R})$). The identity matrix is  denoted by $\operatorname{Id}$.  All vectors are represented in matrix form as column vectors.  

 In what follows, for $1\leq i,j\leq n$, we let $\delta_{ij}$  denote the usual Kronecker delta function:
$$
\delta_{ij}:=\begin{cases}
1 & \mbox{if $i=j$}\\
0 & \mbox{otherwise.}
\end{cases}
$$

\subsection{Remarks on the general approach}
 
In \cite{CF20},  the  authors introduced  a general approach for establishing stability of functional inequalities based on   approximate integration-by-parts identities and Stein's method; the ideas can also be found  in the one-dimensional results of \cite{Ut89, CPU94}. Interested readers are referred to \cite{CF20} for an overview of the  method in abstract settings. We follow this general approach here.  Namely, the first step of the proof (Subsection \ref{subsec:step1}) is dedicated to proving an approximate Stein identity, which mimics the integration by parts formula for the Gaussian measure on Hermite polynomials of degree 2.  The second step (Subsection \ref{subsec:step2}) describes our implementation of Stein's method that yields the main stability result.
 
 \subsection{An approximate Stein identity}\label{subsec:step1}
 For a differentiable matrix-valued function $V: \mathbb{R}^n \to M_{n\times n}(\mathbb{R})$ and $x\in \mathbb{R}^n$, we write $(x \nabla^T) \cdot V$ to denote the scalar product of the operator $(x \nabla^T)$ and the function $V$; that is, 
 $$
  (x \nabla^T) \cdot V = \sum_{i,j=1}^n [x \nabla^T]_{ij} [V]_{ij}. := \sum_{i,j=1}^n x_i \partial_j [V]_{ij}.
 $$
  The following lemma is the main result of this section. 

\begin{lemma}[Approximate Stein Identity]\label{lem:ApproxStein}
Let $\mu$ satisfy the moment assumption.  If $V:   \mathbb{R}^n \to M_{n\times n}(\mathbb{R})$ is a matrix-valued function such that each coordinate $[V]_{ij}:\mathbb{R}^n \to \mathbb{R}$ is integrable and differentiable, with $\|\nabla [V]_{ij}\|_{L^2(\mu)}<\infty$, then 
\begin{align*}
\left| \int   (x x^T-\operatorname{Id}) \cdot V d\mu - 
  \int  (x \nabla^T) \cdot V d\mu \right|  \leq K \sqrt{   C_{PK}(C_{PK}-1)} \sum_{i,j=1}^n \|\nabla [V]_{ij} \|_{L^2(\mu)}, 
    \end{align*}
for $K := (1 + 10/ \sqrt{3}) < 7$.
\end{lemma}
To provide some perspective, we remark that the classical  Stein identity 
$$\int x \cdot \phi d\gamma = \int \operatorname{Id}\cdot  \nabla \phi d\gamma, ~~~ \phi: \mathbb{R}^n\to \mathbb{R}^n
$$
applied to the test function $\phi(x) = V^T(x) x$   gives
\begin{align*}
&  \int   (x x^T-\operatorname{Id}) \cdot V  d\gamma =
 \int   (x \nabla^T) \cdot V d\gamma. 
\end{align*}
Under certain moment assumptions, this consequence of the classical Stein identity also characterizes the Gaussian measure, so may therefore be regarded as one of many ``Stein identities".   It is for this reason that we refer to  Lemma  \ref{lem:ApproxStein}  as an ``approximate Stein identity", which becomes more faithful as $C_{PK}(\mu)$ approaches 1. 

The proof of the approximate Stein identity rests on an approximate integration by parts formula enjoyed by near-extremizers of the Poincar\'e--Korn inequality.  To develop it, we assume henceforth that $\mu$ has finite second moments.  
 For a  function $u \in \mathcal{C}$, define  
  $$
 a_u := \arg\min_{a\in \mathcal{A} }\|u - a \|_{L^2(\mu)}. 
 $$
 Since $\mathcal{A}$ is a closed linear subspace of $L^2(\mu)$ and $\mathcal{C}\subset L^2(\mu)$, the function $a_u$ exists and is unique.  Moreover, as a projection onto $\mathcal{A}$, the map $u \mapsto a_u$ is linear and equal to identity on $\mathcal{A}$.  
 
 
 The following is an approximate integration by parts formula satisfied by near-extremizers of the Poincar\'e--Korn inequality; it does not require the moment assumption, and may therefore be of independent interest.
 \begin{lemma}\label{lem:approxEL} Let $\mu$ be a centered Borel probability measure on $\mathbb{R}^n$ with finite second moments, and $C_{PK}(\mu)<\infty$.   
 If $\epsilon \geq 0$ and $u \in \mathcal{C}$ satisfy
 \begin{align}
 (2-(\epsilon/2)^2) C_{PK}(\mu) \|\nabla_{s}u\|_{L^2(\mu)}^2 \leq \inf_{a\in \mathcal{A}}\|u - a \|^2_{L^2(\mu)}, \label{eq:epsilonExtremal}
 \end{align}
then for every $v \in \mathcal{C}$, we have
 $$\left| \int ( u-a_u)\cdot ( v-a_v)d\mu  - 2 C_{PK}(\mu) \int (\nabla_s u) \cdot (\nabla_s v) d\mu
 \right| \leq \epsilon   C_{PK}(\mu)      \|\nabla_{s}u\|_{L^2(\mu)}  \|\nabla_{s}v\|_{L^2(\mu)}.
 $$
 \end{lemma}
We briefly remark that if $u\in \mathcal{C}$ is an extremizer in the Poincar\'e--Korn inequality, then we have the following (exact) integration by parts formula:
 $$
 \int ( u-a_u)\cdot ( v-a_v)d\mu  = 2 C_{PK}(\mu) \int (\nabla_s u) \cdot (\nabla_s v) d\mu, ~~\forall v\in \mathcal{C}.
 $$
However, we see no reason to expect that nontrivial extremizers exist in the Poincar\'e--Korn inequality for  general   $\mu$, which motivates the approximation in Lemma \ref{lem:approxEL}. Note that there are examples of measures for which the classical Poincar\'e inequality has no non-trivial extremal function, including for example the exponential measure, see \cite[Section 4.4.1]{BGL14}.
 \begin{proof}
 We'll abbreviate $C_{PK}:=C_{PK}(\mu)$ for convenience.    Begin by defining the quotient space $\mathcal{Q} := \mathcal{C}\slash \mathcal{A}$.  
 Note that $\mathcal{A}\subset \ker(\nabla_{s})$, so we may define a linear operator $D$ on $\mathcal{Q}$ via
 $$
 D[u] := \nabla_{s}u, ~~~u\in \mathcal{C},
 $$
 where $[u]\in \mathcal{Q}$ denotes the coset of $u$.  Observe that 
 $$
 \langle [u],[v]\rangle := \int D[u] \cdot D[v]d\mu, ~~[u],[v]\in \mathcal{Q},
 $$
 defines an inner product on $\mathcal{Q}\times \mathcal{Q}$.  Bilinearity and symmetry are self-evident, and positive-definiteness follows from the Poincar\'e--Korn inequality, which states
 $$
 \|[u]\|^2:= \langle [u],[u]\rangle \geq \frac{1}{2C_{PK}}\inf_{a\in \mathcal{A}}\|u - a \|^2_{L^2(\mu)}.
 $$
 Since the quantity on the right is the quotient norm, positive-definiteness follows.  Hence, we are justified in defining a Hilbert space $\mathcal{H}$ as the completion of $\mathcal{Q}$ in the norm $\|\cdot\|$, and extending the inner product $\langle \cdot,\cdot\rangle$  to  $\mathcal{H}$.  
 
Now, define the operator $T: q \in \mathcal{Q} \mapsto Tq$ by 
$$ 
T[u] :=u - a_u, ~~~u \in  \mathcal{C}.
 $$
 The operator $T$ is linear, and well-defined on $\mathcal{Q}$ since $u = a_{u}$ for $u\in \mathcal{A}$.  
Now, fix $u\in \mathcal{C}$.
The   operator
 $$
 [v] \in \mathcal{Q} \mapsto \int T [u] \cdot T [v] d\mu 
 $$ 
 is a  bounded linear operator on $\mathcal{Q}$.  Indeed, boundedness follows by the Cauchy--Schwarz and  Poincar\'e--Korn inequalities as
 $$
 \int T [u] \cdot T [v] d\mu \leq 2C_{PK} \|[u]\|\|[v]\|, ~~\mbox{for all~} v\in \mathcal{C}.
 $$
 Linearity now follows by linearity of $T$ and the integral.  By density of $\mathcal{Q}$ in $\mathcal{H}$ and the Riesz representation theorem, there is $h\in \mathcal{H}$ with $\|h \|\leq \|[u]\|$ such that 
 \begin{align*}
\int T [u] \cdot T [v] d\mu = 2 C_{PK} \langle h, [v]\rangle, ~~~\mbox{for all~} v\in \mathcal{C}.
 \end{align*}
 Hence, for any $v\in \mathcal{C}$, we have 
 \begin{align*}
 \int T [u] \cdot T [v] d\mu - 2 C_{PK}(\mu)  \langle [u],[v]\rangle &= 2 C_{PK}(\mu)  \langle h - [u],[v]\rangle\\
 &\leq 2 C_{PK}(\mu)  \| h - [u]\|  \|[v]\|.
 \end{align*}
Opening the square, we have 
\begin{align*}
  C_{PK} \| h - [u]\|^2 &=   C_{PK} \| h\|^2 - 2   C_{PK} \langle h , [u]\rangle  +   C_{PK} \|[u]\|^2
  \\
   &\leq 2   C_{PK}  \|[u]\|^2 -   \int T [u] \cdot T [u] d\mu\\
   &=2   C_{PK}   \|\nabla_{s}u\|_{L^2(\mu)}^2 -  \inf_{a\in \mathcal{A}}\|u - a \|^2_{L^2(\mu)}.
\end{align*}
The claim follows.  
 \end{proof}

\begin{proposition}\label{prop:MomentImplications}
Let $\mu$ satisfy the moment assumption.
\begin{enumerate}[i)]
\item For each $u \in \mathcal{C}$, we have $a_u(x) =  A_u x$, with 
$$
A_u :=\frac{1}{2} \int\left( u  x^T - x  u^T\right) d\mu.
$$
\item Fix $i,j\in \{1,\dots, n\}$ with $i\neq j$.  The function $u=(u_1, \dots, u_n) \in \mathcal{C}$ defined by 
\begin{align}
u_k(x) =  
\delta_{ik}(1-x_j^2) + \delta_{jk} x_i x_j, ~~~1\leq k \leq n \label{eq:candidateu}
\end{align}
satisfies $a_u = 0$ and 
$$
[\nabla_s u]_{k\ell} =  (\delta_{j\ell}  \delta_{jk})x_i 
- (\delta_{ik}\delta_{j\ell} + \delta_{i\ell}\delta_{jk})\frac{x_j}{2}, ~~1\leq k,\ell\leq n.
$$
\item We have $C_{PK}(\mu)\geq 1$.
\end{enumerate}
\end{proposition}
\begin{proof}
$i)$ Let $A= -A^T$.  By the cyclic property of trace and  the isotropic condition in the moment assumption, we may compute 
$$
\int  (A_u x) \cdot (A x)d\mu = \Tr(A^T_u A)= \frac{1}{2} \int\Tr\left( x  u^TA-u  x^T A  \right) d\mu = \int u\cdot (Ax) d\mu, 
$$
where we used antisymmetry of $A$ in the last step.    It follows that 
$$
\int (u - A_u x)\cdot a d\mu=0, ~~\forall a\in \mathcal{A}.
$$
An application of the Hilbert projection theorem proves $i)$. 

\vskip1ex

\noindent$ii)$ Fix $i,j\in \{1,\dots, n\}$ with $i\neq j$.  For $u$ given by \eqref{eq:candidateu}, we use $i)$ to evaluate
\begin{align*}
2 [A_u]_{k \ell} &=  \int\left( u_k x_{\ell}- x_k u_{\ell}\right) d\mu\\
&= \int \left( \delta_{ik}(1-x_j^2)x_{\ell} + \delta_{jk} x_i x_j x_{\ell} - \delta_{i\ell}(1-x_j^2)x_{k} - \delta_{j\ell} x_i x_j x_{k}\right)d\mu,
\end{align*}
which vanishes by the moment assumption\footnote{Some simple casework shows that, under the assumption that $\mu$ is centered and isotropic, $a_u=0$ for every choice of $i,j$ if and only if all (mixed) third moments vanish.}.  Next, let $\partial_{k}$ denote partial derivative with respect to $x_k$, and observe 
\begin{align*}
2 [\nabla_s u]_{k\ell} &=  \partial_{k} u_{\ell}  + \partial_{\ell} u_k  \\
&= \delta_{i\ell}\partial_{k}(1-x_j^2) + \delta_{j\ell} \partial_{k} x_i x_j 
+ \delta_{ik}\partial_{\ell}(1-x_j^2) + \delta_{jk} \partial_{\ell} x_i x_j\\
&=\delta_{i\ell}\delta_{jk}(-2x_j) + \delta_{j\ell} (\delta_{ik}x_j + \delta_{jk}x_i) 
+ \delta_{ik}\delta_{j\ell}(-2x_j) + \delta_{jk}(\delta_{i\ell}x_j + \delta_{j\ell}x_i) \\
&=   2 \delta_{j\ell}  \delta_{jk} x_i 
- (\delta_{ik}\delta_{j\ell} + \delta_{i\ell}\delta_{jk})x_j .
\end{align*}

\vskip1ex

\noindent$iii)$ For the choice of $u$ given by \eqref{eq:candidateu}, we use $ii)$ and the moment assumption to evaluate
$$
\inf_{a\in \mathcal{A}}\|u - a\|^2_{L^2(\mu)} = \|u\|^2_{L^2(\mu)} = \int x_i^2 x_j^2 d\mu + \int(1-x_j^2)^2 d\mu  = 3
$$
and
$$
2 \|\nabla_s u\|^2_{L^2(\mu)} = 2 \int x_i^2 d\mu + 4 \int (x_j/2)^2 d\mu = 3 . 
$$
It now follows by definitions that $C_{PK}(\mu)\geq 1$. 
 \end{proof}

With  the necessary ingredients established, we turn our attention to the proof of Lemma \ref{lem:ApproxStein}.

\begin{proof}[Proof of Lemma \ref{lem:ApproxStein}]
Abbreviate $C_{PK}:=C_{PK}(\mu)$.  We can assume $C_{PK}<\infty$, else the claim is trivial.  Also, the statement is invariant to adding constants to $V$, so we assume without loss of generality that $\int V d\mu = 0$.

To start, fix $i\neq j$, and let $u$ be given by \eqref{eq:candidateu}.  By the moment assumption,  this choice of $u$ satisfies \eqref{eq:epsilonExtremal} with
\begin{align}
\epsilon = 2 \sqrt{2}\sqrt{1 - \frac{1}{C_{PK}}}. \label{eq:epsValue}
\end{align}

Next, let $\psi: \mathbb{R}^n\to \mathbb{R}$ be  integrable and differentiable, satisfying $\int \psi d\mu = 0$ and $ \|\nabla \psi \|_{L^2(\mu)}<\infty$.   Fix $m\in \{1,\dots, n\}$ and define  $v  := (\delta_{1m}  , \dots, \delta_{nm}  )\psi$.  Since $\int \psi d\mu = 0$ and  $\|\nabla_s v\|_{L^2(\mu)}\leq  \|\nabla \psi \|_{L^2(\mu)}<\infty$, it follows that $v\in \mathcal{C}$.   
For $A_v$ defined as in Proposition \ref{prop:MomentImplications}, we may compute 
\begin{align*}
\int ( u-a_u)&\cdot ( v -a_{v})d\mu = \int u \cdot ( v-a_v)d\mu \\
&= \int u_m \psi d\mu - \sum_{k,\ell} \int u_{k} [A_{v} ]_{k\ell}x_{\ell} d\mu\\
&= \int (\delta_{im}(1-x_j^2)\psi  + \delta_{jm} x_i x_m \psi) d\mu -  \sum_{\ell} \int \left(  [A_v]_{i\ell} (1-x_j^2) +  [A_v]_{j\ell} x_i x_j x_{\ell} \right)d\mu\\
&= \int (\delta_{im}(1-x_j^2)\psi  + \delta_{jm} x_i x_m \psi) d\mu,
\end{align*}
 where the last line follows from the moment assumption.   Next, note that
\begin{align*}
 \int (\nabla_s u) \cdot (\nabla_s v ) d\mu &= \int \left( x_i  [\nabla_s v]_{jj} -x_j [\nabla_s v]_{ij} \right)d\mu\\
 &= \int \left( x_i  \delta_{jm} \partial_j \psi  -\frac{1}{2}x_j [ \delta_{jm}   \partial_i \psi + \delta_{im} \partial_j  \psi ] \right)d\mu.
\end{align*}
Define $\mathcal{E}(\psi) :=2 \sqrt{3} \sqrt{   C_{PK}(C_{PK}-1)}  \|\nabla \psi\|_{L^2(\mu)}$ for convenience. An application of Lemma \ref{lem:approxEL} with $\epsilon$ given in \eqref{eq:epsValue} yields 
\begin{align*}
&\left| \int (\delta_{im}(1-x_j^2)\psi  + \delta_{jm} x_i x_m \psi) d\mu - 
2 C_{PK} \int \left( x_i  \delta_{jm} \partial_j \psi  -\frac{1}{2}x_j [ \delta_{jm}   \partial_i \psi + \delta_{im} \partial_j  \psi ] \right)d\mu\right|  \leq  \mathcal{E}(\psi)
\end{align*}
for all $i,j,m$ with $i\neq j$.  Taking $m=i$ and $\psi = [V]_{jj}$ gives
\begin{align*}
\left| \int (x_j^2-1) [V]_{jj}   d\mu - C_{PK} \int x_j  \partial_j  [V]_{jj}   d\mu \right| \leq  \mathcal{E}([V]_{jj}).
\end{align*}
On the other hand, taking $m = j$ and $\psi=[V]_{ij}$ gives 
\begin{align*}
\left| \int   x_i x_j [V]_{ij} d\mu - 
2 C_{PK} \int \left( x_i   \partial_j [V]_{ij}  -\frac{1}{2}x_j   \partial_i [V]_{ij}   \right)d\mu \right|&\leq  \mathcal{E}([V]_{ij}).
\end{align*}
These can evidently be combined into the single matrix inequality 
\begin{align*}
\left| \int   (x_i x_j-\delta_{ij}) [V]_{ij} d\mu - 
2 C_{PK} \int \left( x_i   \partial_j [V]_{ij}  -\frac{1}{2}x_j   \partial_i [V]_{ij}   \right)d\mu \right|&\leq  \mathcal{E}([V]_{ij}),
\end{align*}
holding for all $1\leq i,j\leq n$.  Summing over all $1\leq i ,j\leq n$ and applying the triangle inequality gives 
\begin{align*}
\left| \int   (x x^T-\operatorname{Id}) \cdot V d\mu - 
2 C_{PK} \int \left( (x \nabla^T) \cdot V
  -\frac{1}{2}  (x \nabla^T)\cdot  V^T    \right)d\mu \right| &\leq  \mathcal{E}(V),
\end{align*}
where $\mathcal{E}(V):= \sum_{i,j=1}^n \mathcal{E}([V]_{ij})$.  The same is true when $V$ is replaced by $V^T$.  However,  the matrix $(xx^T - \operatorname{Id})$ is symmetric, so we have
\begin{align*}
\left| \int   (x x^T-\operatorname{Id}) \cdot V d\mu - 
2 C_{PK} \int \left( (x \nabla^T) \cdot V^T
  -\frac{1}{2}  (x \nabla^T)\cdot  V    \right)d\mu \right| &\leq   \mathcal{E}(V),
\end{align*}
An application of the triangle inequality gives
\begin{align*}
\left|    C_{PK}  \int    (x \nabla^T)\cdot  V  d\mu
- C_{PK}  \int       (x \nabla^T)\cdot  V^T  d\mu \right| \leq  \frac{2}{3}  \mathcal{E}(V),
\end{align*}
and another gives
\begin{align*}
\left| \int   (x x^T-\operatorname{Id}) \cdot V d\mu - 
  C_{PK} \int  (x \nabla^T) \cdot V d\mu \right| &\leq  \frac{5}{3}\mathcal{E}(V).
\end{align*}
Two final applications of the triangle inequality followed by Cauchy--Schwarz gives the desired conclusion 
\begin{align*}
\left| \int   (x x^T-\operatorname{Id}) \cdot V d\mu - 
  \int  (x \nabla^T) \cdot V d\mu \right| 
  &\leq  \frac{5}{3}\mathcal{E}(V) + (C_{PK}-1) \left|  \int  (x \nabla^T) \cdot V d\mu \right| \\
  &\leq  \frac{5}{3}\mathcal{E}(V) + (C_{PK}-1) \sum_{i,j=1}^n \int \left|  x_i \partial_j [V]_{ij}  \right| d\mu\\
  &\leq \frac{5}{3}\mathcal{E}(V) + (C_{PK}-1) \sum_{i,j=1}^n \| \partial_j [V]_{ij}  \|_{L^2(\mu)} \\
   &\leq K \sqrt{   C_{PK}(C_{PK}-1)} \sum_{i,j=1}^n \|\nabla [V]_{ij} \|_{L^2(\mu)} .
    \end{align*}
\end{proof}

\subsection{Implementation of Stein's method}\label{subsec:step2}

For a sufficiently smooth function $f: \mathbb{R}^n \to \mathbb{R}^m$, let $D^k f$ denote the tensor of $k$-th order derivatives.  The tensor $D^kf(x)$ can be regarded as a vector in a space of dimension $m\times n^k$, which we equip with its natural Euclidean norm $\|\cdot\|_2$.

We'll need the following Lemma.  It combines Barbour's solution to the classical Stein equation in terms of the Ornstein--Uhlenbeck semigroup $(P_t)_{t\geq 0}$, defined by 
$$
P_t f(x):= \int_{\mathbb{R}^n} f(e^{-t}x + (1-e^{-2t})^{1/2} z) d\gamma(z), ~~~~f\in L^1(\gamma),
$$
 and the higher-order regularity estimate, that was for example derived in \cite{Fat21}.  
\begin{lemma}
For $f : \mathbb{R}^n \to \mathbb{R}$ with $\int f d\gamma <\infty$, the  function $\varphi_f : \mathbb{R}^n \to \mathbb{R}^n$ defined by 
\begin{align}
\varphi_f (x) := \nabla \int_{0}^{\infty} P_t f(x) dt \label{eq:BarbourSoln}
\end{align}
solves the Poisson equation 
\begin{align}
f - \int f d\gamma = x \cdot \varphi_f - \Tr(\nabla \varphi_f), \label{eq:SteinPDE}
\end{align}
and satisfies 
\begin{align}
\sup_x \| D^{k} \varphi_f(x) \|_2 \leq   \sup_x \| D^{k}  f(x) \|_2, ~~k\geq 1. \label{eq:regularityBds}
\end{align}
\end{lemma}

We are now ready to implement Stein's method to prove our main result. In particular, for a given test function $f: \mathbb{R}^n \to \mathbb{R}$ with uniformly bounded second derivatives, we'll bootstrap the solution $\varphi_f$ to the Stein equation \eqref{eq:SteinPDE} to construct a solution $V: \mathbb{R}^n \to M_{n\times n}(\mathbb{R})$ to the integrated second-order Stein equation
\begin{align*}
\int f d\mu  - \int f d\gamma &= \int \left( (x x^T-\operatorname{Id}) \cdot V -  (x \nabla^T) \cdot V   \right) d\mu.
\end{align*}
The main stability result will then follow from the approximate Stein identity of Lemma \ref{lem:ApproxStein}, regularity estimates on $V$, and definition of the Zolotarev distance.

\begin{proof}[Proof of Theorem \ref{thm:MainResult}] Fix any $f: \mathbb{R}^n \to \mathbb{R}$ satisfying  $\sup_x \|\nabla^2 f(x)\|_2 \leq 1$.  
Using the classical identity $\nabla (P_t f) = e^{-t} P_t (\nabla f)$ for $P_t$, we have
$$
\nabla \varphi_f (x) = \int_0^{\infty} \nabla^2 P_t f(x) dt = \int_0^{\infty}  e^{-2t} P_t( \nabla^2 f)(x) dt. 
$$
By the triangle  and Jensen inequalities, \eqref{eq:regularityBds}, and boundedness of $\nabla^2 f$,    \eqref{eq:BarbourSoln} implies
\begin{align}
\sup_{x}\|\nabla \varphi_f(x)\|_2\leq \frac{1}{2} \sup_x \|\nabla^2 \varphi_f(x)\|_2 \leq \frac{1}{2}  \sup_x \|\nabla^2 f(x)\|_2 \leq \frac{1}{2}. \label{eq:gradientBounds}
\end{align}
In particular, $\varphi_f$ is $1/2$-Lipschitz and $\nabla \varphi_f$ is 1-Lipschitz.  Now, define $a\in \mathbb{R}^n$ and  $Q \in M_{n\times n}(\mathbb{R})$ by
$$
a :=   \varphi_f(0), ~~~[Q]_{ij} := (1+\delta_{ij})[\nabla \varphi_f(0)]_{ij},
$$
and put 
$$
g(x):= f(x) - a^T x - \frac{1}{2}x^T Q x. 
$$
Using the fact that Hermite polynomials are eigenfunctions of the Ornstein--Uhlenbeck semigroup, we can check that the solution $\varphi_g := \nabla \int_{0}^{\infty} P_t g dt$ to the Poisson equation 
\begin{align}
g - \int g d\gamma = x \cdot \varphi_g - \Tr(\nabla \varphi_g), \label{eq:gPoissEqn0}
\end{align}
is equal to 
$$
\varphi_g(x) = \varphi_f(x) -\varphi_f(0)- \nabla \varphi_f(0) x, ~~~x\in \mathbb{R}^n.
$$
In particular, $\varphi_g$  satisfies
\begin{align}
\varphi_g(0) = 0,~~\nabla \varphi_g(0) = 0,  ~~\mbox{and}~~\nabla^2 \varphi_g(x) = \nabla^2 \varphi_f(x). \label{eq:varphig_data}
\end{align}
We now establish some basic regularity properties of $\varphi_g$.  
Combining \eqref{eq:varphig_data} with the Lipschitz estimates established for $\varphi_f$, we have 
$$
|\varphi_g(x)|\leq \frac{1}{2}|x|, ~~\mbox{and}~~\| \nabla \varphi_g(x) \|_2 \leq |x|.
$$
Additionally, by a  Taylor expansion around $x=0$, the properties \eqref{eq:varphig_data} together with boundedness of second-derivatives of $\varphi_f$ imply the  quadratic growth estimate
$$
|\varphi_g(x)| \leq  \frac{1}{2}\|\nabla^2 \varphi_f(0) \|_2 |x|^2 \leq \frac{1}{2}|x|^2.
$$

Next, we define a matrix-valued function $V: \mathbb{R}^n \to M_{n\times n}(\mathbb{R})$ by 
$$
V(x) := \begin{cases}
\frac{1}{|x|^2} x \varphi_g^T(x) & \mbox{if $x\neq 0$}\\
0 &  \mbox{if $x=0$}.
\end{cases}
$$
By definition of $V$ and the fact that $\varphi_g(0)=0$, we have
$$
V^T(x) x = \varphi_g(x), ~~x\in \mathbb{R}^n.
$$
Now, we check the regularity of $V$.   Since $\varphi_g$ inherits continuity properties from $\varphi_f$, it follows that $V$ is continuous on $\mathbb{R}^n\setminus\{0\}$.  It is also continuous at $x=0$, which follows since $|\varphi_g(x)|\leq \frac{1}{2}|x|^2 $, and therefore  $\lim_{x\to 0}V(x) = 0 = V(0)$.

Evidently, $V$ is differentiable on $\mathbb{R}^n\setminus\{0\}$. For $x\neq 0$, we compute
\begin{align*}
\nabla( [V(x)]_{ij}) = \nabla  \frac{x_i [\varphi_g(x)]_j}{|x|^2} = -2x  \frac{x_i [\varphi_g(x)]_j}{|x|^4} + \frac{1}{|x|^2} (e_i [\varphi_g(x)]_j  + x_i \nabla  [\varphi_g(x)]_j ).
\end{align*}
Using the regularity estimates on $\varphi_g$, we obtain  
$$
|\nabla( [V(x)]_{ij})| \leq |x_i| + \frac{1}{2} + \frac{|x_i|}{|x|} \leq |x_i| + \frac{3}{2}.
$$
Thus, using the moment assumption, we have
\begin{align}
\|\nabla( [V]_{ij})\|_{L^2(\mu)} &= \left( \int |\nabla( [V(x)]_{ij})|^2 d\mu(x) \right)^{1/2} \notag\\
&\leq 
\left( \int(|x_i| + 3/2)^2 d\mu(x) \right)^{1/2}\leq 5/2. \label{eq:VgradientBound}
\end{align}

Finally, we put everything together to obtain
\begin{align}
\int f d\mu  - \int f d\gamma &= \int gd\mu - \int g d\gamma \label{eq:mugammaMoments}\\
&=  \int \left( x \cdot \varphi_g - \Tr(\nabla \varphi_g) \right) d\mu  \label{eq:gPoissEqn}\\
&=  \int \left( (x x^T-\operatorname{Id}) \cdot V -  (x \nabla^T) \cdot V   \right) d\mu  \label{eq:applyDefV}\\
&\leq K \sqrt{   C_{PK}(C_{PK}-1)} \sum_{i,j=1}^n \|\nabla [V]_{ij} \|_{L^2(\mu)}  \label{eq:applyApproxStein}\\
&< 20 n^2 \sqrt{   C_{PK}(C_{PK}-1)}. \label{eq:VgradientBound2}
\end{align}
In the above, \eqref{eq:mugammaMoments} follows by definition of $g$ and the moment assumption; 
\eqref{eq:gPoissEqn} is \eqref{eq:gPoissEqn0}, integrated with respect to $\mu$; \eqref{eq:applyDefV} follows since $V^T(x)x = \varphi_g(x)$; \eqref{eq:applyApproxStein} follows from the approximate Stein identity of Lemma \ref{lem:ApproxStein} applied to (smooth approximations of) $V$; \eqref{eq:VgradientBound2} is the bound \eqref{eq:VgradientBound}.  Taking supremum over $f$ proves the theorem.
\end{proof}

\appendix

\section{Proof of Proposition \ref{prop_conj_triviale}}

We first recall Caffarelli's contraction theorem \cite{Caf01}: if a probability measure $d\mu = e^{-\phi}dx$ satisfies $\nabla^2 \phi \geq \operatorname{Id}$, then there exists a transport map $T$ from $\gamma$ onto $\mu$ that is $1$-Lipschitz. This map $T$ is the Brenier map from optimal transport theory. 

 For $T$ as above, we have $|T(x) - T(y)| \leq |x-y|$ for all $x, y \in \R^n$. Moreover, 
\begin{align*}
2n = \int{|x - y|^2 d\mu(x)d\mu(y)} &= \int{|T(x) - T(y)|^2 d\gamma(x)d\gamma(y)} \\
&\leq \int{|x-y|^2d\gamma(x)d\gamma(y)}  = 2n.
\end{align*}
Hence there is equality throughout, and   $|T(x) - T(y)| = |x-y|$, ~ $\gamma^{\otimes 2}$-a.s. Since $T$ is continuous, the equality holds everywhere. In the case where $\mu$ has full support, $T$ is surjective, and therefore $T$ is a surjective isometry. By the Mazur--Ulam theorem, we conclude that $T$ is affine. Since $\mu$ is centered and isotropic by assumption,  it must be standard Gaussian. 

If $\mu$ does not have a full support, we can take a convolution with a standard Gaussian, rescaled so that the new measure $\nu$ is still isotropic. Since $\nu$ has full support and  $1$-uniform log-concavity is preserved by this operation, we can apply the previous case to deduce that $\nu$ is Gaussian. Since a convolution of two measures is Gaussian iff both are Gaussian, it follows that $\mu$ is also Gaussian.


\begin{thebibliography}{99}
\bibitem{AH} B. Arras and C. Houdr\'e, On Stein’s method for multivariate self-decomposable laws with finite first moment. \textit{Electron. J. Probab.} 24: 1--33 (2019).

\bibitem{BGL14} D. Bakry, I. Gentil and M. Ledoux, Analysis and Geometry of Markov Diffusion Operators.  Springer International Publishing,  Springer Cham.  2014.

\bibitem{BH00} N. Belili and H. Heinich, Distances de Wasserstein et de Zolotarev. C. R.
Acad. Sci. Paris, t. 330, Srie I, p. 811--814, 2000.

\bibitem{BU} A.A. Borovkov and S.A. Utev. On an inequality and a related characterisation
of the normal distribution. Theory of Probability and Its Applications, 28:219--
228, 1984.

\bibitem{CPU94} T. Cacoullos, V. Papathanasiou and S. A. Utev, Variational Inequalities with Examples and an Application to the Central Limit Theorem. \textit{Ann. Probab.} 22(3): 1607-1618 (1994).

\bibitem{Caf01} L. Caffarelli, Monotonicity Properties of Optimal Transportation and the FKG and Related Inequalities. \textit{Comm. Math. Phys.} 214, 547--563, (2000). 

\bibitem{CDHMM} K. Carrapatoso, J. Dolbeault, F. H\'erau, S. Mischler and C. Mouhot, Weighted Korn and Poincaré-Korn inequalities in the Euclidean space and associated operators. 
\textit{Arch. Rational Mech. Anal. }243 (3), 1565--1596, 2022.

\bibitem{CMS} F. Cavalletti, A. Mondino and D. Semola, Quantitative Obata's theorem, \textit{Anal. PDE} 16 (2023), no. 6, 1389-1431.

\bibitem{CFM} G. C\'ebron, M. Fathi and T. Mai, A note on existence of free Stein kernels, \textit{Proc. Amer. Math. Soc.} 148 (2020), no. 4, 1583-1594. 

\bibitem{CZ17} X. Cheng and D. Zhou, Eigenvalues of the drifted Laplacian on complete metric measure spaces. \textit{Commun. Contemp.
Math.} Vol. 19, No. 01, 1650001 (2017). 


\bibitem{CF20} T. Courtade and M. Fathi, Stability of the Bakry-Emery theorem on $R^n$. \textit{J. Funct. Anal.} 279 (2020), no. 2, 108523, 28 pp. 

\bibitem{CFP18} T. Courtade, M. Fathi and A. Pananjady, Existence of Stein Kernels under a Spectral Gap, and Discrepancy Bounds, \textit{Ann. IHP: Probab. Stat.}, 55, 2, 2019. 

\bibitem{DPF17} G. De Philippis and A. Figalli, Rigidity and stability of Caffarelli's log-concave perturbation theorem. 
\textit{Nonlinear Anal.} 154 (2017), 59--70. 

\bibitem{Fat21} M. Fathi, Higher-order Stein kernels for Gaussian approximation. \textit{Studia Math.} 256 (2021), no. 3, 241-258. 

\bibitem{FGS} M. Fathi, I. Gentil and J. Serres, Stability estimates for the sharp spectral gap bound under a curvature-dimension condition. To appear in \textit{Ann. Institut Fourier} 

\bibitem{Fig13} A. Figalli, Stability in geometric and functional inequalities
European Congress of Mathematics, 585-599, Eur. Math. Soc., Zurich, 2013.

\bibitem{Fig23} A. Figalli, A short review on improvements and stability for some interpolation inequalities
Proceedings of ICIAM 2023.

\bibitem{Hor95} Horgan, C. O., Korn's Inequalities and Their Applications in Continuum Mechanics. \textit{SIAM Review}, vol. 37, no. 4, 1995, pp. 491--511. 

\bibitem{Kor06} Korn, A. Die Eigenschwingungen eines elastischen K\"orpers mit ruhender Oberfl\"ache. Akad. der Wissensch., Munich,
Math. phys. KI. 36 (1906), 351.

\bibitem{LM16} Lewicka, M., and M\"uller, S. On the optimal constants in Korn’s and geometric rigidity estimates, in bounded and
unbounded domains, under Neumann boundary conditions. Indiana Univ. Math. J. 65, 2 (2016), 377--397.

\bibitem{MO} C.H. Mai and S. Ohta, Quantitative estimates for the Bakry-Ledoux isoperimetric inequality, \textit{Comment. Math. Helv.} 96 (2021), 693-739


\bibitem{Ser23} J. Serres, Stability of the Poincaré constant, \textit{Bernoulli} 29 (2) 1297 - 1320, May 2023.

\bibitem{Ser23b} J. Serres, Stability for higher order eigenvalues in dimension one. \textit{Stoch. Proc. Appl.}, 155, 2023, 459-484. 

\bibitem{Ut89} S. A. Utev. Probabilistic problems connected with an integro-differential inequality. \textit{Sibirsk. Mat.
Zh.}, 30(3):182–186, 220, 1989
\end{thebibliography}
\end{document}